\documentclass[a4paper,12pt]{article}
\usepackage[cp1251]{inputenc}
\usepackage[russian]{babel}
\usepackage{amsfonts, amssymb, amsmath, amsthm, amscd}
\usepackage{cite}
\author{A.A. Vasil'eva}
\title{Widths of weighted Sobolev classes with weights that are
functions of distance to some $h$-set: some limiting cases}
\date{}
\begin{document}

\maketitle

\newenvironment{Biblio}{%
                  \renewcommand{\refname}{\footnotesize REFERENCES}%
                  \begin{thebibliography}{99}%
                  \itemsep -0.2em}%
                  {\end{thebibliography}}

\def\inff{\mathop{\smash\inf\vphantom\sup}}
\renewcommand{\le}{\leqslant}
\renewcommand{\ge}{\geqslant}
\newcommand{\sgn}{\mathrm {sgn}\,}
\newcommand{\inter}{\mathrm {int}\,}
\newcommand{\dist}{\mathrm {dist}}
\newcommand{\supp}{\mathrm {supp}\,}
\newcommand{\R}{\mathbb{R}}
\renewcommand{\C}{\mathbb{C}}
\newcommand{\Z}{\mathbb{Z}}
\newcommand{\N}{\mathbb{N}}
\newcommand{\Q}{\mathbb{Q}}
\theoremstyle{plain}
\newtheorem{Trm}{Theorem}
\newtheorem{trma}{Theorem}

\newtheorem{Def}{Definition}
\newtheorem{Cor}{Corollary}
\newtheorem{Lem}{Lemma}
\newtheorem{Rem}{Remark}
\newtheorem{Sta}{Proposition}
\newtheorem{Sup}{Assumption}
\newtheorem{Exa}{Example}
\renewcommand{\proofname}{\bf Proof}
\renewcommand{\thetrma}{\Alph{trma}}

\section{Introduction}

Let $X$, $Y$ be sets, and let $f_1$, $f_2:\ X\times Y\rightarrow
\mathbb{R}_+$. We write $f_1(x, \, y)\underset{y}{\lesssim} f_2(x,
\, y)$ (or $f_2(x, \, y)\underset{y}{\gtrsim} f_1(x, \, y)$) if
for any $y\in Y$ there exists $c(y)>0$ such that $f_1(x, \, y)\le
c(y)f_2(x, \, y)$ for any $x\in X$; $f_1(x, \,
y)\underset{y}{\asymp} f_2(x, \, y)$ if $f_1(x, \, y)
\underset{y}{\lesssim} f_2(x, \, y)$ and $f_2(x, \,
y)\underset{y}{\lesssim} f_1(x, \, y)$.

Let $\Omega \subset \R^d$ be a bounded domain (i.e., a bounded
open connected set), and let $g$, $v:\Omega\rightarrow (0, \,
\infty)$ be measurable functions. For each measurable
vector-valued function $\psi:\ \Omega\rightarrow \R^l$,
$\psi=(\psi_k) _{1\le k\le l}$, and for any $p\in [1, \, \infty]$
we set
$$
\|\psi\|_{L_p(\Omega)}= \Big\|\max _{1\le k\le l}|\psi _k |
\Big\|_p=\left(\int \limits_\Omega \max _{1\le k\le l}|\psi _k(x)
|^p\, dx\right)^{1/p}.
$$
Let $\overline{\beta}=(\beta _1, \, \dots, \, \beta _d)\in
\Z_+^d:=(\N\cup\{0\})^d$, $|\overline{\beta}| =\beta _1+
\ldots+\beta _d$. For any distribution $f$ defined on $\Omega$ we
write $\displaystyle \nabla ^r\!f=\left(\partial^{r}\! f/\partial
x^{\overline{\beta}}\right)_{|\overline{\beta}| =r}$ (here partial
derivatives are taken in the sense of distributions), and denote
by $l_{r,d}$ the number of components of the vector-valued
distribution $\nabla ^r\!f$. We set
$$
W^r_{p,g}(\Omega)=\left\{f:\ \Omega\rightarrow \R\big| \; \exists
\psi :\ \Omega\rightarrow \R^{l_{r,d}}\!:\ \| \psi \|
_{L_p(\Omega)}\le 1, \, \nabla ^r\! f=g\cdot \psi\right\}
$$
\Big(we denote the corresponding function $\psi$ by
$\displaystyle\frac{\nabla ^r\!f}{g}$\Big),
$$
\| f\|_{L_{q,v}(\Omega)}{=}\| f\|_{q,v}{=}\|
fv\|_{L_q(\Omega)},\qquad L_{q,v}(\Omega)=\left\{f:\Omega
\rightarrow \R| \; \ \| f\| _{q,v}<\infty\right\}.
$$
We call the set $W^r_{p,g}(\Omega)$ a weighted Sobolev class.
Observe that $W^r_{p,1}(\Omega)=W^r_p(\Omega)$ is a non-weighted
Sobolev class. For properties of weighted Sobolev spaces and their
generalizations, we refer the reader to the books \cite{triebel,
kufner, edm_trieb_book, triebel1, edm_ev_book, myn_otel} and the
survey paper \cite{kudr_nik}.

Let $(X, \, \|\cdot\|_X)$ be a normed space, let $X^*$ be its
dual, and let ${\cal L}_n(X)$, $n\in \Z_+$, be the family of
subspaces of $X$ of dimension at most $n$. Denote by $L(X, \, Y)$
the space of continuous linear operators from $X$ into a normed
space $Y$. Also, by ${\rm rk}\, A$ denote the dimension of the
image of an operator $A\in L(X, \, Y)$, and by $\| A\|
_{X\rightarrow Y}$, its norm.

By the Kolmogorov $n$-width of a set $M\subset X$ in the space
$X$, we mean the quantity
$$
d_n(M, \, X)=\inf _{L\in {\cal L}_n(X)} \sup_{x\in M}\inf_{y\in
L}\|x-y\|_X,
$$
by the linear $n$-width, the quantity $$\lambda _n(M, \, X) =\inf
_{A\in L(X, \, X), \, {\rm rk} A\le n}\sup _{x\in M}\| x-Ax\| _X,
$$
and by the Gelfand $n$-width, the quantity
$$
d^n(M, \, X)=\inff _{x_1^*, \, \dots, \, x_n^*\in X^*} \sup
\{\|x\|:\; x\in M, \, x^*_j(x)=0, \; 1\le j\le n\}=
$$
$$
=\inff _{A\in L(X, \, \R^n)}\sup \{\|x\|:\; x\in M\cap \ker A\}.
$$

In estimating Kolmogorov, linear, and Gelfand widths we set,
respectively, $\vartheta_l(M, \, X)=d_l(M, \, X)$ and $\hat q=q$,
$\vartheta_l(M, \, X)=\lambda_l(M, \, X)$ and $\hat q=\min\{q, \,
p'\}$, $\vartheta_l(M, \, X)=d^l(M, \, X)$ and $\hat q=p'$.

In the 1960-1980s problems concerning the values of the widths of
function classes in $L_q$ and of finite-dimensional balls $B_p^n$
in $l_q^n$ were intensively studied. Here $l_q^n$ $(1\le q\le
\infty)$ is the space $\R^n$ with the norm
$$
\|(x_1, \, \dots , \, x_n)\| _q\equiv\|(x_1, \, \dots , \, x_n)\|
_{l_q^n}= \left\{
\begin{array}{l}(| x_1 | ^q+\dots+ | x_n | ^q)^{1/q},\text{ if
}q<\infty ,\\ \max \{| x_1 | , \, \dots, \, | x_n |\},\text{ if
}q=\infty ,\end{array}\right .
$$
$B_p^n$ is the unit ball in $l_p^n$. For more details, see
\cite{tikh_nvtp, itogi_nt, kniga_pinkusa}.

Let us formulate the result on widths of non-weighted Sobolev
classes on a cube in the space $L_q$. We set
\begin{align}
\label{tpqrd_def} \theta_{p,q,r,d}=\left\{ \begin{array}{l}
\frac{\delta}{d}-\left(\frac 1q-\frac 1p\right)_+, \quad
\text{if}\quad p\ge q \quad\text{or}\quad \hat q\le 2,
\\ \min \Bigl\{\frac{\delta}{d}+\min \bigl\{\frac 1p-\frac 1q,
\, \frac 12-\frac{1}{\hat q}\bigr\}, \, \frac{\hat q
\delta}{2d}\Bigr\}, \quad \text{if}\quad p<q, \; \hat
q>2.\end{array}\right.
\end{align}

\begin{trma}
\label{sob_dn} {\rm (see, e.g., \cite{birm, de_vore_sharpley,
bib_kashin, bibl6}).} Let $r\in \N$, $1\le p, \, q\le \infty$,
$\displaystyle \frac rd +\frac 1q-\frac 1p>0$. In addition, we
suppose that
\begin{align}
\label{rdp_kolm} \frac{\delta}{d} +\min \left\{\frac 1p-\frac 1q,
\, \frac 12-\frac{1}{\hat q}\right\}\ne \frac{\hat q \delta}{2d}
\end{align}
in the case $p<q$ and $\hat q>2$. Then
$$
\vartheta_n(W^r_p([0, \, 1]^d), \, L_q([0, \, 1]^d))
\underset{r,d,p,q}{\asymp}n^{-\theta_{p,q,r,d}}.
$$
\end{trma}

The problem concerning estimates of widths of weighted Sobolev
classes in weighted $L_q$-space was studied by Birman and Solomyak
\cite{birm}, El Kolli \cite{el_kolli}, Triebel \cite{triebel,
tr_jat}, Mynbaev and Otelbaev \cite{myn_otel}, Boykov \cite{boy_1,
boy_2}, Lizorkin and Otelbaev \cite{liz_otel1, otelbaev}, Aitenova
and Kusainova \cite{ait_kus1, ait_kus2}. For details, see, e.g.,
\cite{vas_sing}.

Let $|\cdot|$ be a norm on $\R^d$, and let $E, \, E'\subset \R^d$,
$x\in \R^d$. We set
$$
{\rm diam}_{|\cdot|}\, E=\sup \{|y-z|:\; y, \, z\in E\}, \;\; {\rm
dist}_{|\cdot|}\, (x, \, E)=\inf \{|x-y|:\; y\in E\}.
$$

\begin{Def}
\label{fca} Let $\Omega\subset\R^d$ be a bounded domain, and let
$a>0$. We say that $\Omega \in {\bf FC}(a)$ if there exists a
point $x_*\in \Omega$ such that, for any $x\in \Omega$, there
exist a number $T(x)>0$ and a curve $\gamma _x:[0, \, T(x)]
\rightarrow\Omega$ with the following properties:
\begin{enumerate}
\item $\gamma _x\in AC[0, \, T(x)]$, $\left|\frac{d\gamma _x(t)}{dt}\right|=1$ a.e.,
\item $\gamma _x(0)=x$, $\gamma _x(T(x))=x_*$,
\item $B_{at}(\gamma _x(t))\subset \Omega$ for any $t\in [0, \, T(x)]$.
\end{enumerate}
\end{Def}

\begin{Def}
We say that $\Omega$ satisfies the John condition (and call
$\Omega$ a John domain) if $\Omega\in {\bf FC}(a)$ for some $a>0$.
\end{Def}
For a bounded domain the John condition coincides with the
flexible cone condition (see definition in \cite{besov_il1}).
Reshetnyak \cite{resh1, resh2} found an integral representation
for functions on a John domain $\Omega$ in terms of their
derivatives of order $r$. This representation yields that for
$\frac rd-\left(\frac 1p-\frac 1q\right)_+\ge 0$ (for $\frac
rd-\left(\frac 1p-\frac 1q\right)_+> 0$, respectively) the class
$W^r_p(\Omega)$ can be continuously (respectively, compactly)
imbedded into $L_q(\Omega)$ (i.e., the conditions of continuous
and compact imbeddings are the same as for $\Omega=[0, \, 1]^d$).
Moreover, in \cite{vas_john, besov_peak_width} it was proved that
if $\Omega$ is a John domain and $p$, $q$, $r$, $d$ are such as in
Theorem \ref{sob_dn}, then widths have the same orders as for
$\Omega=[0, \, 1]^d$.

Throughout we suppose that $\overline{\Omega} \subset \left(-\frac
12, \, \frac 12\right)^d$ (here $\overline{\Omega}$ is the closure
of $\Omega$).

Denote by $\mathbb{H}$ the set of all non-decreasing functions
defined on $(0, \, 1]$.

\begin{Def}
\label{h_set} (see \cite{m_bricchi1}). Let $\Gamma\subset \R^d$ be
a nonempty compact set and $h\in \mathbb{H}$. We say that $\Gamma$
is an $h$-set if there are a constant $c_*\ge 1$ and a finite
countably additive measure $\mu$ on $\R^d$ such that $\supp
\mu=\Gamma$ and
\begin{align}
\label{c1htmu} c_*^{-1}h(t)\le \mu(B_t(x))\le c_* h(t)
\end{align}
for any $x\in \Gamma$ and $t\in (0, \, 1]$.
\end{Def}

Throughout we suppose that $1<p\le \infty$, $1\le q<\infty$, $r\in
\N$, $\delta:=r+\frac dq-\frac dp>0$. We denote $\log x:=\log_2
x$.

Let $\Gamma\subset \partial \Omega$ be an $h$-set,
\begin{align}
\label{gx_vrph} g(x)=\varphi_g({\rm dist}_{|\cdot|} (x, \,
\Gamma)),\quad v(x)=\varphi_v ({\rm dist}_{|\cdot|}(x, \,
\Gamma)),
\end{align}
where $\varphi_g$, $\varphi_v:(0, \, \infty)\rightarrow (0, \,
\infty)$. Suppose that in some neighborhood of zero
\begin{align}
\label{def_h1} h(t)=t^{\theta}|\log t|^\gamma\tau(|\log t|), \quad
0< \theta<d,
\end{align}
\begin{align}
\label{ghi_g01} \varphi_g(t)=t^{-\beta_g}|\log t|^{-\alpha_g}
\rho_g(|\log t|), \quad \varphi_v(t)=t^{-\beta_v} |\log
t|^{-\alpha_v} \rho_v(|\log t|),
\end{align}
where $\rho_g$, $\rho_v$, $\tau$ are absolutely continuous
functions,
\begin{align}
\label{ll1} \lim \limits _{y\to +\infty}\frac{y\tau'(y)}{\tau(y)}=
\lim \limits _{y\to +\infty}\frac{y\rho_g'(y)}{\rho_g(y)}=\lim
\limits _{y\to +\infty}\frac{y\rho_v'(y)}{\rho_v(y)}=0.
\end{align}

For $\beta_v<\frac{d-\theta}{q}$, in \cite{vas_vl_raspr,
vas_vl_raspr2, vas_width_raspr} there were obtained sufficient
conditions for embedding of $W^r_{p,g}(\Omega)$ into
$L_{q,v}(\Omega)$, and order estimates of Kolmogorov, Gelfand and
linear widths were found. Here we consider the limiting case
\begin{align}
\label{beta_v_eq_dtq} \beta_v=\frac{d-\theta}{q}, \quad
\alpha_v>\frac{1-\gamma}{q}.
\end{align}

We set $\beta=\beta_g+\beta_v$, $\alpha=\alpha_g+\alpha_v$,
$\rho(y)=\rho_g(y)\rho_v(y)$, $\mathfrak{Z}=(p, \, q, \, r, \, d,
\, a, \, c_*, \, h, \, g, \, v)$, $\mathfrak{Z}_*=(\mathfrak{Z},
\, {\rm diam}\, \Omega)$.

\begin{Trm}
\label{main} There exists $n_0=n_0(\mathfrak{Z})$ such that for
any $n\ge n_0$ the following assertion holds.
\begin{enumerate}
\item Let $\beta-\delta+\theta\left(\frac 1q-\frac
1p\right)_+<0$. We set $$\sigma_*(n)=(\log n)^{-\alpha+\frac
1q+\frac{(\beta-\delta)\gamma}{\theta}}\rho(\log
n)\tau^{\frac{\beta-\delta}{\theta}}(\log n).$$
\begin{itemize}
\item Let $p\ge q$ or $p< q$, $\hat q\le 2$.
We set
\begin{align}
\label{case1_theta_j_0} \theta_1=\frac{\delta}{d}-\left(\frac
1q-\frac 1p\right)_+, \quad  \theta_2=\frac{\delta
-\beta}{\theta}-\left(\frac 1q-\frac 1p\right)_+,
\end{align}
\begin{align}
\label{case1_sigma_j_0} \sigma_1(n)=1, \quad
\sigma_2(n)=\sigma_*(n).
\end{align}
Suppose that $\theta_1\ne \theta_2$, $j_*\in \{1, \, 2\}$,
$$
\theta_{j_*}=\min\{\theta_1, \, \theta_2\}.
$$
Then
$$
\vartheta_n(W^r_{p,g}(\Omega), \, L_{q,v}(\Omega))
\underset{\mathfrak{Z}_*}{\asymp}
n^{-\theta_{j_*}}\sigma_{j_*}(n).
$$
\item Let $p<q$, $\hat q>2$. We set
\begin{align}
\label{case1qg2_th12} \theta_1=\frac{\delta}{d} + \min\left\{\frac
1p-\frac 1q, \, \frac 12-\frac{1}{\hat q}\right\}, \quad
\theta_2=\frac{\hat q\delta}{2d},
\end{align}
\begin{align}
\label{case1qg2_th34} \theta_3=\frac{\delta-\beta}{\theta} +
\min\left\{\frac 1p-\frac 1q, \, \frac 12-\frac{1}{\hat
q}\right\}, \quad \theta_4 = \frac{\hat q(\delta
-\beta)}{2\theta},
\end{align}
\begin{align}
\sigma_1(n)=\sigma_2(n)=1, \quad \sigma_3(n)
=\sigma_4(n)=\sigma_*(n).
\end{align}
Suppose that there exists $j_*\in \{1, \, 2, \, 3, \, 4\}$ such
that
\begin{align}
\label{thj_min_case1} \theta_{j_*}<\min_{j\ne j_*} \theta_j.
\end{align}
Then
$$
\vartheta_n(W^r_{p,g}(\Omega), \, L_{q,v}(\Omega))
\underset{\mathfrak{Z}_*}{\asymp}
n^{-\theta_{j_*}}\sigma_{j_*}(n).
$$
\end{itemize}

\item Let $\beta-\delta+\theta\left(\frac 1q-\frac
1p\right)_+=0$. In addition, we suppose that
$\alpha_0:=\alpha-\frac 1q>0$ for $p<q$ and
$\alpha_0:=\alpha-1-(1-\gamma)\left(\frac 1q-\frac 1p\right)>0$
for $p\ge q$. Then
$$
\vartheta_n(W^r_{p,g}(\Omega), \, L_{q,v}(\Omega))
\underset{\mathfrak{Z}_*}{\asymp} (\log n)^{-\alpha_0} \rho(\log
n)\tau^{-\left(\frac 1q-\frac 1p\right)_+}(\log n).
$$
\end{enumerate}
\end{Trm}

\begin{Rem}
\label{non_w} From Theorem \ref{sob_dn} it follows that for
$\frac{\delta-\beta}{\theta}>\frac{\delta}{d}$ the order estimates
are the same as in the non-weighted case.
\end{Rem}

\begin{Rem} \label{rem_nonweight}
Formulas in Theorem \ref{main} differ from formulas in
\cite{vas_width_raspr} by the power of the logarithmic factor.
\end{Rem}

The upper estimates follow from the general result about the
estimate of widths of function classes on sets with tree-like
structure. Problems on estimating widths and entropy numbers for
embedding operators of weighted function classes on trees were
studied in papers of Evans, Harris, Lang, Solomyak, Lifshits and
Linde \cite{ev_har_lang, solomyak, lifs_m, l_l, l_l1}.

Without loss of generality, as $|\cdot|$ we may take $|(x_1, \,
\dots, \, x_d)|=\max _{1\le i\le d}|x_i|$.

\section{Proof of the upper estimate}

In this section, we obtain upper estimates for widths in Theorem
\ref{main}.

The following lemma was proved in \cite{vas_bes} (see inequalities
(60)).
\begin{Lem}
\label{sum_lem1} Let $\Lambda_*:(0, \, \infty) \rightarrow (0, \,
\infty)$ be an absolutely continuous function such that $\lim
\limits_{y\to +\infty}\frac{y\Lambda _*'(y)} {\Lambda _*(y)}=0$.
Then for any $\varepsilon >0$
\begin{align}
\label{rho_yy1} t^{-\varepsilon}
\underset{\varepsilon,\Lambda_*}{\lesssim}
\frac{\Lambda_*(ty)}{\Lambda_*(y)}\underset{\varepsilon,
\Lambda_*}{\lesssim} t^\varepsilon,\quad 1\le y<\infty, \;\; 1\le
t<\infty.
\end{align}
\end{Lem}

Let $c_*\ge 1$ be the constant from the definition of an $h$-set.
From (\ref{def_h1}), (\ref{ghi_g01}), (\ref{ll1}) and Lemma
\ref{sum_lem1} it follows that there exists $c_0=
c_0(\mathfrak{Z})\ge c_*$ such that
\begin{align}
\label{h_cond_1}   \frac{h(t)}{h(s)}\le c_0, \quad
\frac{\varphi_g(t)}{\varphi_g(s)}\le c_0, \quad
\frac{\varphi_v(t)}{\varphi_v(s)}\le c_0, \quad j\in {\mathbb{N}},
\;\; t, \; s\in [2^{-j-1}, \, 2^{-j+1}].
\end{align}

Let $(\Omega, \, \Sigma, \, \nu)$ be a measure space. We say that
sets $A$, $B\subset \Omega$ do not overlap if $\nu(A\cap B)=0$.
Let $m\in {\mathbb{N}}\cup \{\infty\}$, $E$, $E_1, \, \dots, \,
E_m\subset \Omega$ be measurable sets. We say that
$\{E_i\}_{i=1}^m$ is a partition of $E$ if the sets $E_i$ do not
overlap pairwise and $\nu\left[\left(\cup _{i=1}^m
E_i\right)\bigtriangleup E\right]=0$.

Let $({\cal T}, \, \xi_0)$ be a tree rooted at $\xi_0$. We
introduce a partial order on ${\bf V}({\cal T})$ as follows: we
say that $\xi'>\xi$ if there exists a simple path $(\xi_0, \,
\xi_1, \, \dots , \, \xi_n, \, \xi')$ such that $\xi=\xi_k$ for
some $k\in \overline{0, \, n}$. In this case, we set $\rho_{{\cal
T}}(\xi, \, \xi')=\rho_{{\cal T}}(\xi', \, \xi) =n+1-k$. In
addition, we denote $\rho_{{\cal T}}(\xi, \, \xi)=0$. If
$\xi'>\xi$ or $\xi'=\xi$, then we write $\xi'\ge \xi$ and denote
$[\xi, \, \xi']:= \{\xi''\in {\bf V}({\cal T}):\xi \le \xi''\le
\xi'\}$. This partial order on ${\cal T}$ induces a partial order
on its subtree.

Given $j\in {\mathbb{Z}}_+$, $\xi\in {\bf V}({\cal T})$, we set
$$
\label{v1v}{\bf V}_j(\xi):={\bf V}_j ^{{\cal T}}(\xi):=
\{\xi'\ge\xi:\; \rho_{{\cal T}}(\xi, \, \xi')=j\}.
$$
For each vertex $\xi\in {\bf V}({\cal T})$ we denote by ${\cal
T}_\xi=({\cal T}_\xi, \, \xi)$ a subtree in ${\cal T}$ with vertex
set $\{\xi'\in {\bf V}({\cal T}):\xi'\ge \xi\}$.

In \cite{vas_vl_raspr, vas_vl_raspr2} a tree $({\cal A}, \,
\eta_{j_*,1})$ with vertex set $\{\eta_{j,i}\}_{j\ge j_*, \, i\in
\tilde I_j}$ was constructed, as well as the partition of $\Omega$
into subdomains $\Omega[\xi]$, $\xi\in {\bf V}({\cal A})$.
Moreover, ${\bf V}_{j-j_*}^{{\cal
A}}(\eta_{j_*,1})=\{\eta_{j,i}\}_{i\in \tilde I_j}$ and there
exists a number $\overline{s}=\overline{s}(a, \, d)\in \N$ such
that
$$
{\rm diam}\, \Omega[\eta_{j,i}] \underset{a,d,c_0}{\asymp}
2^{-\overline{s}j}, \quad {\rm dist}_{|\cdot|}\, (x, \, \Gamma)
\underset{a,d,c_0}{\asymp} 2^{-\overline{s}j}, \quad x\in
\Omega[\eta_{j,i}],
$$
$$
{\rm card}\, {\bf V}_{j'-j}^{{\cal A}}(\eta_{j,i})
\underset{a,d,c_0}{\lesssim}
\frac{h(2^{-\overline{s}j})}{h(2^{-\overline{s}j'})}, \quad j'\ge
j\ge j_*.
$$
In particular,
\begin{align}
\label{v1le1} {\rm card}\, {\bf V}_1^{{\cal A}}(\eta_{j,i})
\stackrel{(\ref{h_cond_1})}{\underset{a,d,c_0}{\lesssim}} 1, \quad
j\ge j_*.
\end{align}

We set
\begin{align}
\label{u_w_xi} u(\eta_{j,i})
=u_j=\varphi_g(2^{-\overline{s}j})\cdot 2^{-\left(r-\frac
dp\right)\overline{s}j}, \quad w(\eta_{j,i})
=w_j=\varphi_v(2^{-\overline{s}j})\cdot
2^{-\frac{d\overline{s}j}{q}}.
\end{align}

Given a subtree ${\cal D}\subset {\cal A}$, we denote
$\Omega[{\cal D}] =\cup _{\xi\in {\bf V}({\cal D})} \Omega[\xi]$.

In \cite{vas_sib_m_j} sufficient conditions for embedding
$W^r_{p,g}(\Omega)$ into $L_{q,v}(\Omega)$ were obtained; here
(\ref{h_cond_1}) holds and the functions $g$, $v$ satisfy
(\ref{gx_vrph}). Let us formulate these results.

\begin{trma}
\label{emb_gen_plq_cor} Let $u$, $w$ be defined by (\ref{u_w_xi}),
$1<p<q<\infty$. Suppose that there exist $l_0\in {\mathbb{N}}$ and
$\lambda\in (0, \, 1)$ such that
\begin{align}
\label{lsijl_hh} \frac{\left(\sum \limits _{i=j+l_0}^\infty
\frac{h(2^{-\overline{s}(j+l_0)})}{ h(2^{-\overline{s}i})}w_i^q
\right)^{1/q} } {w_j} \le \lambda, \quad j\ge j_*.
\end{align}
Let $\sup _{j\ge j_*} u_j\left(\sum \limits _{i=j}^\infty
\frac{h(2^{-\overline{s}j})}{
h(2^{-\overline{s}i})}w_i^q\right)^{1/q}<\infty$. Then
$W^r_{p,g}(\Omega) \subset L_{q,v}(\Omega)$ and for any $k\ge
j_*$, $\xi_*\in {\bf V}^{\cal A}_{k-j_*}(\eta_{j_*,1})$ there
exists a linear continuous operator $P:L_{q,v}(\Omega) \rightarrow
{\cal P}_{r-1} (\Omega)$ such that for any subtree ${\cal
D}\subset {\cal A}$ rooted at $\xi_*$ and for any function $f\in
W^r_{p,g}(\Omega)$
$$
\|f-Pf\|_{L_{q,v}(\Omega[{\cal D}])} \underset{\mathfrak{Z}}
{\lesssim} \sup _{j\ge k} u_j \left(\sum \limits _{i\ge j}
\frac{h(2^{-\overline{s}j})}{ h(2^{-\overline{s}i})}w_i^q
\right)^{\frac 1q}\left\|\frac{\nabla^r
f}{g}\right\|_{L_p(\Omega[{\cal D}])}.
$$
\end{trma}

\begin{trma}
\label{emb_gen_pgq} Let $p\ge q$, $\xi_*\in {\bf V}^{\cal
A}_{k-j_*}(\eta_{j_*,1})$, and let the functions $u$, $w$ on ${\bf
V}({\cal A})$ be defined by (\ref{u_w_xi}). We set $\hat
w_j=w_j\cdot \left(\frac{h(2^{-\overline{s}k})}
{h(2^{-\overline{s}j})}\right)^{\frac 1q}$, $\hat u_j=u_j\cdot
\left(\frac{h(2^{-\overline{s}j})} {h(2^{-\overline{s}k})} \right)
^{\frac 1p}$, $k\le j<\infty$. Let
\begin{align}
\label{m_uw_peq} M_{\hat u,\hat w}(k):=\sup _{k\le
j<\infty}\Bigl(\sum \limits _{i=j}^\infty \hat w_i^q\Bigr)^{\frac
1q}\Bigl( \sum \limits _{i=k}^j \hat
u_i^{p'}\Bigr)^{\frac{1}{p'}}< \infty,\quad 1<p= q<\infty,
\end{align}
\begin{align}
\label{m_uw_pgq} M_{\hat u,\hat w}(k):=\left(\sum \limits
_{j=k}^\infty \left(\Bigl(\sum \limits _{i=j}^\infty \hat
w_i^q\Bigr)^{\frac 1p}\Bigl(\sum \limits _{i=k}^j \hat
u_i^{p'}\Bigr)^{\frac{1}{p'}}\right) ^{\frac{pq}{p-q}}\hat
w_j^q\right)^{\frac 1q-\frac 1p}<\infty,\quad q<p.
\end{align}
Then $W^r_{p,g}(\Omega[{\cal A}_{\xi_*}]) \subset
L_{q,v}(\Omega[{\cal A}_{\xi_*}])$ and there exists a linear
continuous operator $P:L_{q,v}(\Omega) \rightarrow {\cal P}_{r-1}
(\Omega)$ such that for any subtree ${\cal D}\subset {\cal
A}_{\xi_*}$ rooted at $\xi_*$ and for any function $f\in
W^r_{p,g}(\Omega)$
$$
\|f-Pf\|_{L_{q,v}(\Omega[{\cal D}])} \underset{\mathfrak{Z}}
{\lesssim} M_{\hat u,\hat w}(k)\left\|\frac{\nabla^r
f}{g}\right\|_{L_p(\Omega[{\cal D}])}.
$$
\end{trma}

Suppose that (\ref{def_h1}), (\ref{ghi_g01}), (\ref{ll1}),
(\ref{beta_v_eq_dtq}) hold.

From (\ref{ghi_g01}), (\ref{beta_v_eq_dtq}) and (\ref{u_w_xi}) it
follows that
\begin{align}
\label{uw_ex}
u(\eta_{j,i})=u_j=2^{\overline{s}j\left(\beta_g-r+\frac
dp\right)}(\overline{s}j)^{-\alpha_g} \rho_g(\overline{s}j), \quad
w(\eta_{j,i})=w_j=2^{-\frac{\theta \overline{s}j} {q}}
(\overline{s}j)^{-\alpha_v} \rho_v(\overline{s}j).
\end{align}

Recall that $\delta =r+\frac dq-\frac dp$.

\begin{Cor}
\label{cor_plq} Let $1<p<q<\infty$, $r\in \N$, $\delta>0$, and let
the conditions (\ref{def_h1}), (\ref{ghi_g01}), (\ref{ll1}),
(\ref{beta_v_eq_dtq}) hold. In addition, we suppose that
\begin{align}
\label{bdl0bde0a} \text{either}\quad \beta-\delta< 0 \quad \text{
or}\quad \beta-\delta=0, \quad \alpha>\frac 1q.
\end{align}
Then $W^r_{p,g}(\Omega) \subset L_{q,v}(\Omega)$ and for any $k\ge
j_*$, $\xi_*\in {\bf V}^{\cal A}_{k-j_*}(\eta_{j_*,1})$ there
exists a linear continuous operator  $P:L_{q,v}(\Omega)
\rightarrow {\cal P}_{r-1} (\Omega)$ such that for any subtree
${\cal D}\subset {\cal A}$ rooted at $\xi_*$ and for any function
$f\in W^r_{p,g}(\Omega)$
$$
\|f-Pf\|_{L_{q,v}(\Omega[{\cal D}])} \underset{\mathfrak{Z}}
{\lesssim} 2^{-(\delta-\beta)\overline{s}k}
(\overline{s}k)^{-\alpha+\frac 1q}
\rho(\overline{s}k)\left\|\frac{\nabla^r
f}{g}\right\|_{L_p(\Omega[{\cal D}])}.
$$
\end{Cor}
\begin{proof}
From (\ref{def_h1}) and (\ref{uw_ex}) it follows that
\begin{align}
\label{sl_ij_inf}
\begin{array}{c}
\sum \limits _{i=j}^\infty \frac{h(2^{-\overline{s}j})}{
h(2^{-\overline{s}i})}w_i^q =\sum \limits _{i=j}^\infty 2^{-\theta
\overline{s}i} (\overline{s}i)^{-\alpha_vq}
\rho_v^q(\overline{s}i)\cdot \frac{2^{\theta
\overline{s}i}(\overline{s}j)^{\gamma}\tau(\overline{s}j)}
{2^{\theta\overline{s}j} (\overline{s}i)^\gamma
\tau(\overline{s}i)}
\stackrel{(\ref{beta_v_eq_dtq}),(\ref{rho_yy1})}{\underset{\mathfrak{Z}}{\asymp}}
\\
\asymp 2^{-\theta\overline{s}j}[\overline{s}j]^{-\alpha_vq+1}
\rho_v^q(\overline{s}j).
\end{array}
\end{align}
This together with Lemma \ref{sum_lem1} implies (\ref{lsijl_hh}).
Further,
$$
\sup _{j\ge k} u_j \left(\sum \limits _{i\ge j}
\frac{h(2^{-\overline{s}j})}{ h(2^{-\overline{s}i})}w_i^q
\right)^{\frac 1q}
\stackrel{(\ref{beta_v_eq_dtq}),(\ref{uw_ex}),(\ref{bdl0bde0a}),(\ref{sl_ij_inf})}
{\underset{\mathfrak{Z}}{\asymp}}
2^{-(\delta-\beta)\overline{s}k}(\overline{s}k)^{-\alpha+\frac 1q}
\rho(\overline{s}k).
$$
It remains to apply Theorem \ref{emb_gen_plq_cor}.
\end{proof}

Let us consider the case $p\ge q$. We apply Theorem
\ref{emb_gen_pgq}. For $j\ge k$, we have
\begin{align}
\label{hat_u} \begin{array}{c} \hat u_j\stackrel{(\ref{def_h1}),
(\ref{uw_ex})}{=} 2^{\overline{s}j \left(\beta_g-r+
\frac{d}{p}\right)}(\overline{s}j) ^{-\alpha_g}
\rho_g(\overline{s}j) \cdot 2^{-\frac{\theta
\overline{s}(j-k)}{p}}\frac{j^{\frac{\gamma}{p}}\tau^{\frac{1}{p}}
(\overline{s}j)}{k^{\frac{\gamma}{p}}\tau^{\frac{1}{p}}(\overline{s}k)},
\\ \hat w_j \stackrel{(\ref{def_h1}), (\ref{uw_ex})}{=}
2^{-\frac{\theta \overline{s}k}{q}}(\overline{s}j) ^{-\alpha_v}
\rho_v(\overline{s}j) \cdot \frac{k^{\frac{\gamma}{q}}
\tau^{\frac{1}{q}} (\overline{s}k)} {j^{\frac{\gamma}{q}}
\tau^{\frac{1}{q}}(\overline{s}j)}.
\end{array}
\end{align}
\begin{Cor}
\label{cor_pgeq} Let $1<p\le\infty$, $1\le q<\infty$, $p\ge q$,
$r\in \N$, $\delta>0$ and let conditions (\ref{def_h1}),
(\ref{ghi_g01}), (\ref{ll1}), (\ref{beta_v_eq_dtq}) hold. Suppose
that either $\beta-\delta+\theta\left(\frac 1q-\frac 1p\right)<0$
or $\beta-\delta+\theta\left(\frac 1q-\frac 1p\right)=0$ and
$\alpha-1-(1-\gamma)\left(\frac 1q-\frac 1p\right)>0$. Then
$W^r_{p,g}(\Omega) \subset L_{q,v}(\Omega)$ and for any $k\ge
j_*$, $\xi_*\in {\bf V}^{\cal A}_{k-j_*}(\eta_{j_*,1})$ there
exists a linear continuous operator $P:L_{q,v}(\Omega) \rightarrow
{\cal P}_{r-1} (\Omega)$ such that for any subtree ${\cal
D}\subset {\cal A}$ rooted at $\xi_*$ and for any function $f\in
W^r_{p,g}(\Omega)$
$$
\|f-Pf\|_{L_{q,v}(\Omega[{\cal D}])} \underset{\mathfrak{Z}}
{\lesssim} 2^{-(\delta-\beta)\overline{s}k}
(\overline{s}k)^{-\alpha+\frac 1q}
\rho(\overline{s}k)\left\|\frac{\nabla^r
f}{g}\right\|_{L_p(\Omega[{\cal D}])}
$$
in the case $\beta-\delta+\theta\left(\frac 1q-\frac 1p\right)<0$,
and
$$
\|f-Pf\|_{L_{q,v}(\Omega[{\cal D}])} \underset{\mathfrak{Z}}
{\lesssim} 2^{-\theta\left(\frac 1q-\frac 1p\right)\overline{s}k}
(\overline{s}k)^{-\alpha+1+\frac 1q-\frac 1p}
\rho(\overline{s}k)\left\|\frac{\nabla^r
f}{g}\right\|_{L_p(\Omega[{\cal D}])}
$$
in the case $\beta-\delta+\theta\left(\frac 1q-\frac 1p\right)=0$.
\end{Cor}
\begin{proof}
Let $p=q$. Applying (\ref{hat_u}) and (\ref{m_uw_peq}) and taking
into account that $\alpha_v>\frac{1-\gamma}{q}$ and $\beta_g-r+
\frac{d}{p}-\frac{\theta}{p}\stackrel{(\ref{beta_v_eq_dtq})}{=}\beta-\delta$,
we get
$$
M_{\hat u,\hat w}(k)\underset{\mathfrak{Z}}{\lesssim}\sup _{l\ge
k} (\overline{s}l) ^{-\alpha_v+ \frac{1-\gamma}{q}}
\rho_v(\overline{s}l) \tau^{-\frac{1}{q}}(\overline{s}l)
\left(\sum \limits _{j=k}^l 2^{p'(\beta-\delta)\overline{s}j}
(\overline{s}j)^{p'\left(-\alpha_g+\frac{\gamma}
{p}\right)}\rho_g^{p'}(\overline{s}j)\tau^{\frac{p'}{p}}
(\overline{s}j)\right)^{\frac{1}{p'}}.
$$
If $\beta-\delta<0$, then by Lemma \ref{sum_lem1}
\begin{align}
\label{mmm} M_{\hat u,\hat w}(k) \underset{\mathfrak{Z}}
{\lesssim} 2^{(\beta-\delta)\overline{s}k}
(\overline{s}k)^{-\alpha+\frac 1q} \rho(\overline{s}k).
\end{align}
Let $\beta-\delta=0$. We may assume that
$-\alpha_g+\frac{\gamma}{p}+\frac{1}{p'}>0$ (otherwise, we
multiply $\hat u_j$ by $\frac{j^c}{k^c}$ with some $c>0$). Then
\begin{align}
\label{mmm1} M_{\hat u,\hat w}(k) \underset{\mathfrak{Z}}
{\lesssim} (\overline{s}k)^{-\alpha+1} \rho(\overline{s}k).
\end{align}

Let $p>q$. Applying (\ref{hat_u}) and (\ref{m_uw_pgq}) and taking
into account that $\alpha_v>\frac{1-\gamma}{q}$ and $\beta_g-r+
\frac{d}{p}-\frac{\theta}{p}\stackrel{(\ref{beta_v_eq_dtq})}{=}\beta-\delta+\theta\left(\frac
1q-\frac 1p\right)$, we get
$$
M_{\hat u,\hat w}(k)
\stackrel{(\ref{beta_v_eq_dtq})}{\underset{\mathfrak{Z}}{\lesssim}}
2^{-\theta \overline{s}k\left(\frac 1q-\frac 1p\right)}
(\overline{s}k)^{\gamma\left(\frac 1q-\frac 1p\right)} \tau^{\frac
1q-\frac 1p}(\overline{s}k)\times
$$
$$
\times \left(\sum \limits_{j=k}^\infty (\overline{s}j)
^{\frac{pq}{p-q}\left(-\alpha_v-\frac{\gamma}{q}+\frac 1p\right)}
\rho_v^{\frac{pq}{p-q}}(\overline{s}j) \tau^{-\frac{p}{p-q}}
(\overline{s}j) \sigma(j)^{\frac{pq}{p-q}}\right)^{\frac 1q-\frac
1p},
$$
where
$$
\sigma(j)=\left(\sum \limits _{i=k}^j
2^{\overline{s}i\left(\beta-\delta+\theta\left(\frac 1q-\frac
1p\right)\right)p' } (\overline{s}i)^{p'\left(-\alpha_g+
\frac{\gamma}{p}\right)} \rho_g^{p'}(\overline{s}i)
\tau^{\frac{p'}{p}}(\overline{s}i) \right)^{\frac{1}{p'}}.
$$
If $\beta-\delta+\theta\left(\frac 1q-\frac 1p\right)<0$, then
$$
\sigma(j) \underset{\mathfrak{Z}}{\lesssim}
2^{\overline{s}k\left(\beta-\delta+\theta\left(\frac 1q-\frac
1p\right)\right)}(\overline{s}k)^{\left(-\alpha_g+
\frac{\gamma}{p}\right)} \rho_g(\overline{s}k)
\tau^{\frac{1}{p}}(\overline{s}k),
$$
and by the second relation in (\ref{beta_v_eq_dtq}) we have
\begin{align}
\label{mmm2} M_{\hat u,\hat w}(k) \underset{\mathfrak{Z}}
{\lesssim} 2^{\left(\beta-\delta\right)\overline{s}k}
(\overline{s}k)^{-\alpha+\frac 1q} \rho(\overline{s}k).
\end{align}
If $\beta-\delta+\theta\left(\frac 1q-\frac 1p\right)=0$ and
$\alpha>1+(1-\gamma)\left(\frac 1q-\frac 1p\right)$, then we may
assume that $-\alpha_g+\frac{\gamma}{p}+ \frac{1}{p'}>0$. We have
\begin{align}
\label{mmm3} M_{\hat u,\hat w}(k) \underset{\mathfrak{Z}}
{\lesssim} 2^{-\theta\left(\frac 1q-\frac 1p\right)\overline{s}k}
(\overline{s}k)^{-\alpha+1+\frac 1q-\frac 1p} \rho(\overline{s}k).
\end{align}
This completes the proof.
\end{proof}

\begin{Rem} Notice that in order to prove Theorems \ref{emb_gen_plq_cor} and
\ref{emb_gen_pgq} we use estimates for norms of summation
operators on trees, which are obtained in \cite{vas_hardy_tree}.
If $\beta-\delta+\theta\left(\frac 1q-\frac 1p\right)_+<0$, then
these estimates can be proved easier (we argue similarly as in
\cite[Lemma 5.1]{vas_vl_raspr2}).
\end{Rem}

Applying Corollaries \ref{cor_plq} and \ref{cor_pgeq} and arguing
similarly as in \cite[Theorem 1]{vas_width_raspr}, we obtain the
desired upper estimate of widths.

\renewcommand{\proofname}{\bf Proof}

\section{Proof of the lower estimate}

In this section, we obtain the lower estimates of widths in
Theorem \ref{main}.

If $\frac{\delta-\beta}{\theta}> \frac{\delta}{d}$, then by
Theorem \ref{sob_dn} (see also Remark \ref{non_w}) and by the
upper estimate of $\vartheta_n(W^r_{p,g}(\Omega), \,
L_{q,v}(\Omega))$, which is already obtained, we have
$\vartheta_n(W^r_{p,g}(\Omega), \, L_{q,v}(\Omega))
\underset{\mathfrak{Z}}{\lesssim} \vartheta_n(W^r_p([0, \, 1]^d),
\, L_q([0 ,\, 1]^d))$. On the other hand, there is a cube
$\Delta\subset \Omega$ with side length $l(\Delta)
\underset{\mathfrak{Z}_*}{\asymp} 1$ such that $g(x)
\underset{\mathfrak{Z}_*}{\asymp} 1$, $v(x) \underset
{\mathfrak{Z}_*}{\asymp} 1$ for any $x\in \Delta$ (see
\cite{vas_width_raspr}). Hence, $\vartheta_n(W^r_{p,g}(\Omega), \,
L_{q,v}(\Omega)) \underset{\mathfrak{Z}_*}{\gtrsim}
\vartheta_n(W^r_p([0, \, 1]^d), \, L_q([0 ,\, 1]^d))$. Thus, we
obtained the order estimates of widths in the case
$\frac{\delta-\beta}{\theta}>\frac{\delta}{d}$.

Consider the case $\frac{\delta-\beta}{\theta}\le
\frac{\delta}{d}$. In order to obtain the lower estimates we argue
similarly as in \cite{vas_width_raspr}. It is sufficient to prove
the following assertions.

\begin{Sta}
\label{low_est_sta1} Let $\frac{\delta-\beta}{\theta}\le
\frac{\delta}{d}$. Suppose that one of the following conditions
holds: 1) $\beta-\delta +\theta\left(\frac 1q-\frac 1p\right)_+<0$
or 2) $\beta=\delta$, $p<q$. Then there exist
$t_0=t_0(\mathfrak{Z}_*)\in \N$ and $\hat k=\hat
k(\mathfrak{Z}_*)\in \N$ such that for any $t\in \N$, $t\ge t_0$
there exist functions $\psi_{j,t}\in C^\infty_0(\R^d)$ $(1\le j\le
j_t)$ with pairwise non-overlapping supports such that
\begin{align}
\label{j_t_card_low_est} j_t\underset{\mathfrak{Z}_*}{\gtrsim}
2^{\theta \hat kt}(\hat kt)^{-\gamma} \tau^{-1}(\hat kt),
\end{align}
\begin{align}
\label{psi_jt_norm_est} \left\|\frac{\nabla^r
\psi_{j,t}}{g}\right\|_{L_p(\Omega)}=1, \quad
\|\psi_{j,t}\|_{L_{q,v}(\Omega)}
\underset{\mathfrak{Z}_*}{\gtrsim} 2^{(\beta-\delta)\hat k t}
(\hat kt)^{-\alpha +\frac 1q} \rho(\hat kt).
\end{align}
\end{Sta}

\begin{Sta}
\label{low_est_sta2} Let $\beta-\delta +\theta\left(\frac 1q-\frac
1p\right)=0$, $p\ge q$. Then there exist
$t_0=t_0(\mathfrak{Z}_*)\in \N$ and $\hat k=\hat
k(\mathfrak{Z}_*)\in \N$ such that for any $t\in \N$, $t\ge t_0$
there exist functions $\psi_{j,t}\in C^\infty_0(\R^d)$ $(1\le j\le
j_t)$ with pairwise non-overlapping supports such that
\begin{align}
\label{j_t_card_low_est1} j_t\underset{\mathfrak{Z}_*}{\gtrsim}
2^{\theta \hat kt}(\hat kt)^{-\gamma} \tau^{-1}(\hat kt),
\end{align}
\begin{align}
\label{psi_jt_norm_est1}
\left\|\frac{\nabla^r\psi_{j,t}}{g}\right\|_{L_p(\Omega)}=1, \quad
\|\psi_{j,t}\|_{L_{q,v}(\Omega)}
\underset{\mathfrak{Z}_*}{\gtrsim} 2^{-\theta\left(\frac 1q-\frac
1p\right)\hat kt}(\hat kt)^{-\alpha +\frac 1q+1-\frac 1p}
\rho(\hat kt).
\end{align}
\end{Sta}

\smallskip

First we formulate the Vitali covering theorem \cite[p.
408]{leoni1}).

\begin{trma}
\label{vitali} Denote by $B(x, \, t)$ the open or closed ball of
radius $t$ with respect to some norm on $\R^d$ centered in $x$.
Let $E \subset \R^d$ be a finite union of balls $B(x_i, \, r_i)$,
$1\le i\le l$. Then there exists a subset ${\cal I}\subset \{1, \,
\dots, \, l\}$ such that the balls $\{B(x_i, \, r_i)\}_{i\in {\cal
I}}$ are pairwise non-overlapping and $E\subset \cup_{i\in {\cal
I}} B(x_i, \, 3r_i)$.
\end{trma}

Let ${\cal K}$ be a family of closed cubes in $\R^d$ with axes
parallel to coordinate axes. Given a cube $K \in {\cal K}$ and
$s\in \Z_+$, we denote by $\Xi _s(K)$ the partition of $K$ into
$2^{sd}$ closed non-overlapping cubes of the same size, and we set
$\Xi(K):=\bigcup_{s\in \Z_+} \Xi _s(K)$.

Given a cube $\Delta\in \Xi\left(\left[-\frac 12, \, \frac
12\right]^d\right)$ such that $\Delta \cap\Gamma \ne \varnothing$,
we define the cubes $Q_\Delta$, $\tilde Q_\Delta$, $\hat Q_\Delta$
and the points $x_\Delta$, $\hat x_\Delta$ as follows.

Let $m\in \N$, $\Delta \in \Xi_m\left(\left[-\frac 12, \, \frac
12\right]^d\right)$, $\Delta\cap \Gamma\ne \varnothing$. We choose
$x_\Delta\in \Delta\cap \Gamma$ and a cube $Q_\Delta$ such that
$\Delta\in \Xi_1(Q_\Delta)$,
\begin{align}
\label{dist_xdel_2m} {\rm dist}_{|\cdot|} (x_\Delta, \, \partial
Q_\Delta)\ge 2^{-m-1}.
\end{align}
Denote by $\hat x_\Delta$ the center of $Q_\Delta$. Then
\begin{align}
\label{h_del_12} Q_\Delta=\hat x_\Delta+2^{-m+1} \cdot\left[-\frac
12, \, \frac 12\right]^d.
\end{align}
We set
\begin{align}
\label{t_del_14} \tilde Q_\Delta=\hat x_\Delta+3\cdot 2^{-m}
\cdot\left[-\frac 12, \, \frac 12\right]^d, \quad \hat
Q_\Delta=\hat x_\Delta+2^{-m+2} \cdot\left[-\frac 12, \, \frac
12\right]^d.
\end{align}

Recall that the norm $|\cdot|$ is defined by $|(x_1, \, \dots, \,
x_d)|=\max _{1\le i\le d}|x_i|$.

Let $\hat k\in \N$ (it will be chosen later). For each $l\in \Z_+$
we set
\begin{align}
\label{el_del_def} \hat E_l(\Delta)=\{x\in \hat Q_\Delta:\; {\rm
dist}_{|\cdot|}(x, \, \Gamma) \le 2^{-m-\hat kl+2}\}, \quad
E_l(\Delta)=\hat E_l(\Delta)\cap Q_\Delta \cap \Omega.
\end{align}
Notice that
\begin{align}
\label{hqd_e0d} \hat Q_\Delta=\hat E_0(\Delta).
\end{align}

Denote by ${\rm mes}\, A$ the Lebesgue measure of the measurable
set $A\subset \R^d$.
\begin{Lem}
\label{mes_e} The following estimate holds:
\begin{align}
\label{mes_el} {\rm mes}\, \hat E_l(\Delta)
\underset{\mathfrak{Z}_*}{\lesssim} 2^{-md-(d-\theta) \hat
kl}\frac{m^\gamma \tau(m)}{(m+\hat kl)^\gamma \tau(m+\hat kl)}.
\end{align}
In addition, there exists $m_0=m_0(\mathfrak{Z}_*)$ such that for
$m\ge m_0$
\begin{align}
\label{mes_el_bel} {\rm mes}\, E_l(\Delta)
\underset{\mathfrak{Z}_*}{\gtrsim} 2^{-md-(d-\theta) \hat
kl}\frac{m^\gamma \tau(m)}{(m+\hat kl)^\gamma \tau(m+\hat kl)}.
\end{align}
\end{Lem}
\begin{proof}
Let us prove (\ref{mes_el}). Consider the covering of the set
$\hat E_l(\Delta)$ by cubes $x+K$, $x\in \Gamma\cap \hat
Q_\Delta$, $K=\bigl(-2^{-m-\hat kl+3}, \, 2^{-m-\hat kl+3}\bigr)$.
We take a finite subcovering; applying Theorem \ref{vitali} (the
balls are taken with respect to $|\cdot|$), we get a family of
pairwise non-intersecting balls $\{x_i+K\}_{i=1}^N$ such that
$\{x_i+3K\}_{i=1}^N$ is a covering of $\hat E_l(\Delta)$. Since
$\cup_{i=1}^N (x_i+K)$ is contained in a ball $B$ of radius
$\tilde R\underset{\mathfrak{Z}_*}{\asymp} 2^{-m}$, we have
$$
\sum \limits _{i=1}^N \mu(x_i+K)\le \mu(B)
\stackrel{(\ref{c1htmu}), (\ref{h_cond_1})} {\underset
{\mathfrak{Z}_*}{\lesssim}} h(2^{-m});
$$
since $x_i\in \Gamma$, we get $\mu(x_i+K)\stackrel{(\ref{c1htmu}),
(\ref{h_cond_1})} {\underset {\mathfrak{Z}_*}{\asymp}}
h(2^{-m-\hat kl})$ and $N\underset {\mathfrak{Z}_*}{\lesssim}
\frac{h(2^{-m})}{h(2^{-m-\hat kl})}$. Finally,
$$
{\rm mes}\, \hat E_l(\Delta) \le \sum \limits _{i=1}^N {\rm mes}\,
(x_i+3K) \underset{\mathfrak{Z}_*}{\lesssim} 2^{-(m+\hat
kl)d}\frac{h(2^{-m})}{h(2^{-m-\hat kl})}.
$$
It remains to apply (\ref{def_h1}).

Let us prove (\ref{mes_el_bel}). Denote by $Q^*_\Delta$ the
homothetic transform of the cube $Q_\Delta$ with respect to its
center with the coefficient $1-2^{-\hat kl-3}$. We set
$$
\{\Delta_i\}_{i=1}^L=\bigl\{\Delta'\in \Xi_{m+\hat
kl+3}\bigl([-1/2, \, 1/2]^d\bigr):\; \Delta'\subset Q^*_\Delta, \;
\Delta'\cap \Gamma\ne \varnothing\bigr\}.
$$
It can be proved similarly as formula (4.20) in
\cite{vas_vl_raspr} that $L\underset{\mathfrak{Z}_*}{\asymp}
\frac{h(2^{-m})}{h(2^{-m-\hat kl})}$. Since $\Delta_i\cap \Gamma
\ne \varnothing$, it follows from the definition of $\Delta_i$ and
$Q^*_\Delta$ that $\cup_{i=1}^L Q_{\Delta_i} \subset \hat
E_l(\Delta)\cap Q_\Delta$. Finally, for any $j\in \{1, \, \dots,
\, L\}$
$$
{\rm card}\, \{i\in \overline{1, \, L}:\; {\rm mes}\,
(Q_{\Delta_i}\cap Q_{\Delta_j})>0\} \underset{d}{\lesssim} 1.
$$
Therefore, it is sufficient to prove that ${\rm mes}
(Q_{\Delta_i}\cap \Omega) \underset{\mathfrak{Z}_*}{\asymp}
2^{-(m+\hat kl)d}$.

Let $x\in Q_{\Delta_i}\cap \Omega$, $|x-x_{\Delta_i}| \le
2^{-m-\hat kl-5}$. This point exists since $x_{\Delta_i}\in \Gamma
\subset\partial \Omega$ and (\ref{dist_xdel_2m}) holds with
$m+\hat kl+3$ instead of $m$; moreover, ${\rm dist}_{|\cdot|}(x,
\,\partial Q_{\Delta_i}){\underset{d}{\gtrsim}} 2^{-m-\hat kl}$.
Let $x_*$ and $\gamma_x(\cdot):[0, \, T(x)] \rightarrow \Omega$ be
such as in Definition \ref{fca}. There exists
$m_0=m_0(\mathfrak{Z}_*)$ such that $x_*\notin Q_{\Delta_i}$ for
$m\ge m_0$. Let $\gamma_x(t_*)\in
\partial Q_{\Delta_i}$. Then $t_*\underset{d}{\gtrsim} 2^{-m-\hat kl}$. By Definition
$\ref{fca}$, the ball $B_{at_*}(\gamma_x(t_*))$ is contained in
$\Omega$. It remains to observe that ${\rm mes}\,
\bigl(B_{at_*}(\gamma_x(t_*))\cap Q_{\Delta_i}\bigr)
\underset{\mathfrak{Z}_*}{\gtrsim} 2^{-(m+\hat kl)d}$.
\end{proof}

\begin{Rem}
\label{mes_el_rem} From (\ref{mes_el}) it follows that ${\rm
mes}\, (\hat Q_\Delta\cap \Gamma)=0$.
\end{Rem}

Suppose that $m\ge m_0(\mathfrak{Z}_*)$.

Choose $\hat k=\hat k(\mathfrak{Z}_*)$ such that for any $l\in
\Z_+$
\begin{align}
\label{mes_ell1} {\rm mes}\, \left(E_l(\Delta)\backslash \hat
E_{l+1}(\Delta)\right) \underset{\mathfrak{Z}_*}{\asymp}
2^{-md-(d-\theta)\hat kl}\frac{m^\gamma \tau(m)}{(m+\hat
kl)^\gamma \tau(m+\hat kl)}
\end{align}
(it is possible by (\ref{rho_yy1}), (\ref{mes_el}) and
(\ref{mes_el_bel})).

Let $\psi \in C_0^\infty(\R^d)$, ${\rm supp}\, \psi\subset
\left[-\frac 12, \, \frac 12\right]^d$, $\psi|_{\left[-\frac 38,
\, \frac 38\right]^d}=1$, $\psi(x)\in [0, \, 1]$ for any $x\in
\R^d$. We set
\begin{align}
\label{psi_del_x_def} \psi_{\Delta}(x)=\psi(2^{m-2}(x-\hat
x_\Delta)).
\end{align}
Then
\begin{align}
\label{supp_psi_del} {\rm supp}\, \psi_{\Delta} \subset \hat
Q_\Delta, \quad \psi_\Delta|_{\tilde Q_\Delta}=1,
\end{align}
\begin{align}
\label{pointw_est} \left|\frac{\nabla^r \psi_\Delta(x)}
{g(x)}\right| \stackrel{(\ref{gx_vrph}), (\ref{ghi_g01}),
(\ref{el_del_def})}{\underset{\mathfrak{Z}_*}{\lesssim}}
2^{-\beta_g(m+\hat kl)} (m+ \hat kl)^{\alpha_g} \rho_g^{-1}(m+\hat
kl) \cdot 2^{rm}, \quad x\in \hat E_l(\Delta) \backslash \hat
E_{l+1}(\Delta).
\end{align}

We set $c_{\Delta}=\left\|\frac{\nabla^r \psi_{
\Delta}}{g}\right\|_{L_p(\hat Q_\Delta)}^{-1}>0$.

\begin{Lem}
\label{c_del} The following estimates hold:
\begin{align}
\label{c_del_1} c_{\Delta} \underset{\mathfrak{Z}_*}{\gtrsim}
2^{\left(\beta_g-r+\frac dp\right)m}m^{-\alpha_g}\rho_g(m), \quad
c_{\Delta} \|\psi_{\Delta}\|_{L_{q,v}(\Omega)}
\underset{\mathfrak{Z}_*}{\gtrsim} 2^{(\beta-\delta)m}
m^{-\alpha+\frac 1q} \rho(m).
\end{align}
\end{Lem}
\begin{proof}
We estimate the value $\left\| \frac{\nabla^r \psi_{
\Delta}}{g}\right\|_{L_p(\hat Q_\Delta)}$ from above. First we
notice that from the conditions $\frac{\delta-\beta}{\theta}\le
\frac{\delta}{d}$, $\theta<d$ and
$\beta_v\stackrel{(\ref{beta_v_eq_dtq})}{=}\frac{d-\theta}{q}$ it
follows that
\begin{align}
\label{bg_dp} \beta_g+\frac{d-\theta}{p}>0.
\end{align}
Hence, by Remark \ref{mes_el_rem},
$$
\left\| \frac{\nabla^r \psi_{\Delta}}{g}\right\| _{L_p(\hat
Q_\Delta)}^p\stackrel{(\ref{hqd_e0d})}{=}\sum \limits _{l\in \Z_+}
\left\| \frac{\nabla^r \psi_{\Delta}}{g}\right\| ^p _{L_p(\hat
E_l(\Delta)\backslash \hat E_{l+1}(\Delta))} \stackrel{
(\ref{mes_el}),
(\ref{pointw_est})}{\underset{\mathfrak{Z}_*}{\lesssim}}
$$
$$
\lesssim\sum \limits _{l\in \Z_+} 2^{-p\beta_g (m+\hat k l)}
(m+\hat k l)^{p\alpha_g} \rho_g^{-p}(m+\hat k l) \cdot
2^{prm}\cdot 2^{-dm-(d-\theta)\hat k l} \frac{m^\gamma \tau(m)}
{(m+\hat kl)^\gamma \tau(m+\hat k l)}
\stackrel{(\ref{bg_dp})}{\underset{\mathfrak{Z}_*}{\asymp}}
$$
$$
\asymp 2^{p\left(-\beta_g+r-\frac dp\right)m} m^{p\alpha_g}\rho
_g^{-p}(m).
$$
This implies the first inequality in (\ref{c_del_1}). Let us prove
the second inequality. Taking into account that
$\psi_{\Delta}|_{Q_\Delta}\stackrel{(\ref{supp_psi_del})}{=}1$ and
$\beta_v=\frac{d-\theta}{q}$, we get
$$
\|\psi_{\Delta}\|^q_{L_{q,v}(\Omega)} \ge \sum \limits _{l\in
\Z_+} \|v\psi_{\Delta}\|^q_{L_q(E_l(\Delta)\backslash \hat
E_{l+1}(\Delta))} \stackrel{(\ref{gx_vrph}),(\ref{ghi_g01}),
(\ref{el_del_def}),(\ref{mes_ell1})}{\underset{\mathfrak{Z}_*}{\gtrsim}}
$$
$$
\gtrsim \sum \limits _{l\in \Z_+} 2^{\beta_vq(m+\hat k l)} (m+\hat
k l)^{-\alpha_v q} \rho_v^q(m+\hat k l) \cdot
2^{-md-(d-\theta)\hat k l}\frac{m^\gamma \tau(m)} {(m+\hat k
l)^\gamma \tau(m+\hat k l)}
\stackrel{(\ref{beta_v_eq_dtq})}{\underset{\mathfrak{Z}_*}{\gtrsim}}
$$
$$
\gtrsim 2^{(\beta_vq-d)m}m^{-q\alpha_v+1} \rho_v^q(m).
$$
It remains to apply the first inequality in (\ref{c_del_1}).
\end{proof}

\renewcommand{\proofname}{\bf Proof of Proposition \ref{low_est_sta1}}

\begin{proof}
Let
\begin{align}
\label{delta_nu_def} \{\Delta_\nu\}_{\nu\in {\cal
N}}=\left\{\Delta\in \Xi_{\hat kt}\left(\left[-\frac 12, \, \frac
12\right]^d\right):\; \Delta\cap \Gamma \ne\varnothing\right\}.
\end{align}
Then $\{\hat Q_{\Delta_\nu}\}_{\nu\in {\cal N}}$ is a covering of
$\Gamma$. Denote by $Q^*_{\Delta_\nu}$ the homothetic transform of
$\hat Q_{\Delta_\nu}$ with respect to its center with coefficient
3. Applying Theorem \ref{vitali}, we get that there exists a
subset ${\cal N}'\subset {\cal N}$ such that $\{\hat
Q_{\Delta_\nu}\}_{\nu\in {\cal N}'}$ are pairwise non-overlapping
and $\{Q^*_{\Delta_\nu}\}_{\nu\in {\cal N}'}$ is a covering of
$\Gamma$. We claim that
\begin{align}
\label{crd_n_pr} {\rm card}\, {\cal N}'
\underset{\mathfrak{Z}_*}{\gtrsim} 2^{\theta \hat kt} (\hat
kt)^{-\gamma}\tau^{-1}(\hat kt).
\end{align}
Indeed,
$$
{\rm card}\, {\cal N}'\cdot 2^{-\theta \hat kt} (\hat kt)
^{\gamma}\tau(\hat kt)\stackrel{(\ref{def_h1})}{=}{\rm card}\,
{\cal N}' \cdot h(2^{-\hat kt})
\stackrel{(\ref{c1htmu}),(\ref{h_cond_1})} {\underset
{\mathfrak{Z}_*}{\asymp}} \sum \limits _{\nu\in {\cal N}'} \mu
(Q_{\Delta_\nu}^*) \ge \mu(\Gamma)
\underset{\mathfrak{Z}_*}{\asymp} 1.
$$

We take $\{c_{\Delta_\nu}\psi_{\Delta_\nu}\}_{\nu\in {\cal N}'}$
as the desired function set. It remains to apply Lemma \ref{c_del}
with $m=\hat kt$ and (\ref{crd_n_pr}).
\end{proof}

\renewcommand{\proofname}{\bf Proof}

Let us prove Proposition \ref{low_est_sta2}. Since
$\beta-\delta+\theta\left(\frac 1q-\frac 1p\right)=0$ and
$\beta_v=\frac{d-\theta}{q}$, then
\begin{align}
\label{beta_g_eq} \beta_g=r-\frac dp+\frac{\theta}{p}.
\end{align}

Let $t\in \N$ be sufficiently large, and let $\Delta \in \Xi_{\hat
kt}\left(\left[-1/2, \, 1/2\right]^d\right)$, $\Delta \cap
\Gamma\ne \varnothing$. For each $s\in \Z_+$ we set
\begin{align}
\label{def_del_s_i} \{\Delta_{s,i}\}_{i\in J_s}=\bigl\{\Delta'\in
\Xi_{\hat k(t+s)}\bigl([-1/2, \, 1/2]^d\bigr):\; \Delta'\subset
\hat Q_\Delta, \;\; \Delta'\cap \Gamma \ne \varnothing\bigr\}.
\end{align}

Let
$$
f_{\Delta}(x)=\sum \limits _{s=0}^t \sum \limits _{i\in J_s}
\psi_{\Delta_{s,i}}(x),
$$
where functions $\psi_{\Delta_{s,i}}$ are defined by formula
similar to (\ref{psi_del_x_def}).

There are a number $t_0=t_0(\mathfrak{Z}_*)$ and a cube
$\Delta_0\in \Xi_{\hat k (t-t_0)}\left(\left[-1/2, \,
1/2\right]^d\right)$ such that $\Delta\subset \Delta_0$,
$\Gamma\cap \Delta_0\ne \varnothing$ and ${\rm supp}\,
f_{\Delta}\subset \hat Q_{\Delta_0}$.

Let $l\in \Z_+$, $x\in \hat E_l(\Delta_0)\backslash \hat
E_{l+1}(\Delta_0)$ (see (\ref{el_del_def}) with $m=\hat
k(t-t_0)$). Then ${\rm dist}_{|\cdot|}\, (x, \, \Gamma)
\underset{\mathfrak{Z}_*}{\asymp} 2^{-\hat k(t+l)}$. We estimate
$\left|\frac{\nabla^r f_{\Delta}(x)}{g(x)}\right|$ from above. If
$x\in {\rm supp}\, \psi_{{\Delta}_{s,i}}$ for some $i\in J_s$,
then
$$
\left|\frac{\nabla^r \psi_{\Delta_{s,i}}(x)}{g(x)}\right|
\stackrel{(\ref{gx_vrph}),(\ref{ghi_g01}),(\ref{el_del_def}),
(\ref{def_del_s_i})}{\underset{\mathfrak{Z}_*}{\lesssim}}
2^{-\beta_g\hat k(t+l)} (\hat k(t+l))^{\alpha_g} \rho_g^{-1}(\hat
k(t+l)) \cdot 2^{r\hat k(t+s)}.
$$
Moreover, by (\ref{def_del_s_i}) we get $s\le l+s_0$ with
$s_0=s_0(\mathfrak{Z}_*)$. Since ${\rm supp}\, \psi_{\Delta_{s,i}}
\stackrel{(\ref{supp_psi_del})}{\subset} \hat Q_{\Delta_{s,i}}$,
by the definition of $\hat Q_{\Delta_{s,i}}$ it follows that for
any $x\in \hat Q_{\Delta_0}$ the inequality ${\rm card}\, \{i\in
J_s:\; x\in {\rm supp}\, \psi_{\Delta_{s,i}}\}
\underset{\mathfrak{Z}_*}{\lesssim} 1$ holds. Hence, for $l\le
t-s_0$
$$
\left|\frac{\nabla^r f_{\Delta}(x)}{g(x)}\right|
\underset{\mathfrak{Z}_*}{\lesssim} \sum \limits _{s=0}^{l+s_0}
2^{-\beta_g\hat k(t+l)} (\hat k(t+l))^{\alpha_g} \rho_g^{-1}(\hat
k(t+l)) \cdot 2^{r\hat k(t+s)}
\stackrel{(\ref{rho_yy1})}{\underset {\mathfrak{Z}_*}{\lesssim}}
$$
$$
\lesssim 2^{(r-\beta_g) \hat k(t+l)}(\hat kt)^{\alpha_g}
\rho_g^{-1}(\hat kt),
$$
and for $l>t-s_0$
$$
\left|\frac{\nabla^r f_{\Delta}(x)}{g(x)}\right|
\underset{\mathfrak{Z}_*}{\lesssim}  \sum \limits _{s=0}^t
2^{-\beta_g\hat k(t+l)} (\hat k(t+l))^{\alpha_g} \rho_g^{-1}(\hat
k(t+l)) \cdot 2^{r\hat k(t+s)}
\stackrel{(\ref{rho_yy1})}{\underset {\mathfrak{Z}_*}{\lesssim}}
$$
$$
\lesssim 2^{-\beta_g\hat k(t+l)} (\hat k(t+l))^{\alpha_g}
\rho_g^{-1}(\hat k(t+l))\cdot 2^{2r\hat k t}.
$$
This yields that
$$
\left\|\frac{\nabla^r f_{\Delta}}{g}\right\|^p_{L_p(\Omega)}
\stackrel{(\ref{hqd_e0d})}{=}\sum \limits _{l=0}^\infty
\left\|\frac{\nabla^r f_{\Delta}}{g}\right\|^p_{L_p(\hat
E_l(\Delta_0)\backslash \hat E_{l+1}(\Delta_0))}
\stackrel{(\ref{mes_el})}{\underset{\mathfrak{Z}_*}{\lesssim}}
$$
$$
\lesssim \sum \limits _{l=0}^{t-s_0} 2^{p(r-\beta_g) \hat
k(t+l)}(\hat kt)^{p\alpha_g} \rho_g^{-p}(\hat kt) \cdot 2^{-\hat
ktd-(d-\theta) \hat kl}\frac{(\hat kt)^\gamma \tau(\hat kt)}{(\hat
k(t+l))^\gamma \tau(\hat k(t+l))}+
$$
$$
+ \sum \limits _{l=t-s_0+1}^\infty 2^{-p\beta_g\hat k(t+l)} (\hat
k(t+l))^{p\alpha_g} \rho_g^{-p}(\hat k(t+l))\cdot 2^{2r\hat k tp}
\cdot 2^{-\hat ktd-(d-\theta) \hat kl}\frac{(\hat kt)^\gamma
\tau(\hat kt)}{(\hat k(t+l))^\gamma \tau(\hat k(t+l))}
\stackrel{(\ref{beta_g_eq})}{\underset{\mathfrak{Z}_*}{\lesssim}}
$$
$$
\lesssim 2^{-\hat kt\theta} (\hat kt) ^{\alpha_gp+1}
\rho_g^{-p}(\hat kt).
$$
Thus,
\begin{align}
\label{n_fg} \left\|\frac{\nabla^r f_{\Delta}}{g}
\right\|_{L_p(\Omega)} \underset{\mathfrak{Z}_*}{\lesssim}
2^{-\frac{\hat kt\theta}{p}} (\hat kt)^{\alpha_g+\frac
1p}\rho_g^{-1}(\hat kt).
\end{align}

Let us estimate $\|f_{\Delta}\|_{L_{q,v}(\Omega)}$ from below. Let
$x\in E_l(\Delta)\backslash \hat E_{l+1}(\Delta)$. Then ${\rm
dist}_{|\cdot|} (x, \, \Gamma)
\stackrel{(\ref{el_del_def})}{\underset{\mathfrak{Z}_*}{\asymp}}
2^{-\hat k(t+l)}$ and there exists $l_0=l_0(\mathfrak{Z}_*)$ such
that for $0\le s\le l-l_0$ there exists $i_s\in J_s$ such that
$x\in \tilde Q _{\Delta_{s,i_s}}$. (Indeed, since $x\in Q_\Delta$
by (\ref{el_del_def}), there exists a point $y\in \Gamma\cap \hat
Q_\Delta$ such that $|x-y|\underset{\mathfrak{Z}_*}{\asymp}
2^{-\hat k(t+l)}$. We choose a cube $\Delta_{s,i_s}$ that contains
the point $y$. By the definition of the cube $Q_{\Delta_{s,i_s}}$,
we have $\hat x_{\Delta_{s,i_s}}\in \Delta_{s,i_s}$; hence,
$|y-\hat x_{\Delta_{s,i_s}}|\stackrel{(\ref{def_del_s_i})} {\le}
2^{-\hat k(t+s)}$. Therefore, $|x-\hat x_{\Delta_{s,i_s}}| \le
|x-y|+|y-\hat x_{\Delta_{s,i_s}}|\le c(\mathfrak{Z}_*)2^{-\hat
k(t+l)}+2^{-\hat k(t+s)}$ for some $c(\mathfrak{Z}_*)>0$. It
remains to apply (\ref{t_del_14}), (\ref{def_del_s_i}) and the
inequality $s\le l-l_0$.) Hence, for $\frac{t}{2}\le l\le t$ we
have $|f_{\Delta}(x)| \stackrel{(\ref{supp_psi_del})}
{\underset{\mathfrak{Z}_*}{\gtrsim}} t$. Consequently,
$$
\|f_{\Delta}\|^q_{L_{q,v}(\Omega)} \ge \sum \limits _{t/2\le l\le
t}\|f_{\Delta}\|^q_{L_{q,v}(E_l(\Delta)\backslash \hat
E_{l+1}(\Delta))} \stackrel{(\ref{gx_vrph}),(\ref{ghi_g01}),
(\ref{el_del_def}),(\ref{mes_ell1})}{\underset{\mathfrak{Z}_*}{\gtrsim}}
$$
$$
\gtrsim \sum \limits _{t/2\le l\le t} t^q\cdot 2^{\beta_vq\hat
k(t+l)}(\hat k(t+l))^{-\alpha_v q} \rho_v^q(\hat k(t+l))\cdot
2^{-\hat ktd-(d-\theta) \hat k l}\frac{(\hat kt)^\gamma \tau(\hat
kt)}{(\hat k(t+l))^\gamma \tau(\hat k(t+l))}
\stackrel{(\ref{beta_v_eq_dtq})}{\underset{\mathfrak{Z}_*}{\gtrsim}}
$$
$$
\gtrsim 2^{-\theta\hat kt} (\hat kt)^{-\alpha_vq+q+1}
\rho_v^q(\hat kt);
$$
i.e.,
\begin{align}
\label{f_del_q} \|f_{\Delta}\|_{L_{q,v}(\Omega)}
\underset{\mathfrak{Z}_*}{\gtrsim} 2^{-\frac{\theta\hat kt}{q}}
(\hat kt)^{-\alpha_v+1+\frac 1q} \rho_v(\hat kt).
\end{align}

\renewcommand{\proofname}{\bf Proof of Proposition \ref{low_est_sta2}}
\begin{proof}
Let the set of cubes $\{\Delta_\nu\}_{\nu\in {\cal N}}$ be defined
by formula (\ref{delta_nu_def}), and let
$F_{\Delta_\nu}=c_{\Delta_\nu}f_{\Delta_\nu}$, with
$c_{\Delta_\nu}$ such that $\left\|\frac{\nabla^r
F_{\Delta_\nu}}{g}\right\|_{L_p(\Omega)}=1$. From (\ref{n_fg}) and
(\ref{f_del_q}) it follows that
$$
\|F_{\Delta_\nu}\|_{L_{q,v}(\Omega)}
\underset{\mathfrak{Z}_*}{\gtrsim} 2^{-\theta\left(\frac 1q-\frac
1p\right)\hat kt}(\hat kt)^{-\alpha +\frac 1q+1-\frac 1p}
\rho(\hat kt).
$$
Further, ${\rm supp}\, F_{\Delta_\nu}={\rm supp}\,
f_{\Delta_\nu}\subset \hat Q_{(\Delta_\nu)_0}$ and ${\rm diam}\,
\hat Q_{(\Delta_\nu)_0} \underset{\mathfrak{Z}_*}{\asymp} 2^{-\hat
kt}$. We apply Theorem \ref{vitali} to the covering $\{\hat
Q_{(\Delta_\nu)_0}\}_{\nu\in {\cal N}}$ of the set $\Gamma$ and
argue similarly as in the proof of Proposition \ref{low_est_sta1}.
\end{proof}

\begin{Rem}
Let $\beta_v=\frac{d-\theta}{q}$, $\beta-\delta+\theta\left(\frac
1q-\frac 1p\right)_+=0$. In addition, let $\alpha<\frac 1q$ in the
case $1<p<q<\infty$, and let $\alpha<1+(1-\gamma)\left(\frac
1q-\frac 1p\right)$ in the case $p\ge q$. Then Propositions
\ref{low_est_sta1} and \ref{low_est_sta2} hold; it implies that
$\vartheta_n(W^r_{p,g}(\Omega), \, L_{q,v}(\Omega))=\infty$ for
any $n\in \Z_+$. In particular, if we take $\vartheta_n=d_n$, then
we get that the deviation of $W^r_{p,g}(\Omega)$ from any
finite-dimensional subspace is infinite.
\end{Rem}

\renewcommand{\proofname}{\bf Proof}

\end{document}